\documentclass[12pt]{amsart}
\usepackage{amsmath}
\usepackage{amssymb}
\usepackage[mathcal]{eucal}
\usepackage[all]{xy}
\usepackage{latexsym}
\usepackage{amstext}
\usepackage{amsfonts}
\usepackage{amsthm}
\usepackage{amsopn}
\usepackage{amsbsy}
\usepackage{layout}
\usepackage{color}
\usepackage{graphicx}
\usepackage{bm}
\usepackage{yfonts}
\usepackage{pifont}
\usepackage{tikz}

\theoremstyle{plain}
\newtheorem{theorem}{Theorem}[section]
\newtheorem{proposition}[theorem]{Proposition}
\newtheorem{lemma}[theorem]{Lemma}

\theoremstyle{definition}
\newtheorem{definition}[theorem]{Definition}
\newtheorem{problem}[theorem]{Problem}
\theoremstyle{remark}

\newtheorem{remark}[theorem]{Remark}

\def\nn{\|\cdot\|}

\def\ee{\varepsilon}

\def\xx{\overline x}

\def\N{{\mathbb N}}
\def\R{{\mathbb R}}
\def\B{{\mathbb B}}

\def\vv{\varphi}
\def\ee{\varepsilon}

\def\ee{\varepsilon}

\def\nn{\|\cdot\|}

\date{}

\begin{document}
\title[A survey on the Asplund property for spaces  $C(X)$]{A survey on the Asplund property for spaces  $C(X)$ of continuous functions}
\author{Marian Fabian}
\address{Institute of Mathematics, Czech Academy of Sciences, Prague, Czech Republic}
\email{fabian@math.cas.cz}
\author{Jerzy K{\c{a}}kol}
\address{Faculty of Mathematics and Informatics, A. Mickiewicz University,
61-614 Pozna\'{n}, Poland}
\email{kakol@amu.edu.pl}
\author{Arkady Leiderman}
\address{Department of Mathematics, Ben-Gurion University of the Negev, Beer Sheva, P.O.B. 653, Israel}
\email{arkady@math.bgu.ac.il}

\begin{abstract}
The Asplund property plays an important role in Banach space theory due to its connections with differentiability properties of continuous convex functions, optimization problems, and the weak topology of Banach spaces.
Motivated by the variety of nonequivalent definitions proposed in the literature for locally convex spaces, in this survey we provide a unified framework for studying the Asplund property beyond the Banach space setting.

 The main object of our study is the locally convex space of continuous functions $C_k(X)$ endowed with the compact-open topology, where $X$ is an arbitrary Tychonoff space.
We completely characterize the Asplund property for $C_k(X)$ in terms of topological properties of the underlying space $X$.
 As an essential step, we revisit the proof of several classical results, including the Namioka--Phelps theorem. Our approach is independent of differentiability techniques and relies solely on topological methods.

The exposition is self-contained, and all major results are provided with complete proofs, making the paper accessible to both specialists and newcomers.

\end{abstract}

\bigskip

\maketitle
\section{Introduction}

The study of differentiability properties of convex functions in infinite-dimensional spaces has a long and distinguished history. As noted by R.R. Phelps in his seminal monograph \cite{p}, the first fundamental result in this direction was established by S. Mazur in 1933 (see \cite{Mazur}).

S. Mazur proved that every continuous convex function $f:D\to\mathbb{R}$ defined on an open convex subset $D$ of a separable Banach space $E$ is G\^ateaux differentiable at all points of a dense $G_\delta$ subset of $D$. This theorem marks the beginning of the modern theory of generic differentiability in infinite-dimensional analysis.

A major breakthrough came more than three decades later, in 1968, when E. Asplund \cite{Asplund} substantially strengthened Mazur's theorem.
He extended this result in two ways: first, by proving the same statement for a broader range of Banach spaces, and second, by identifying a more restricted class (now called Asplund spaces)
in which a stronger conclusion of Fr\'echet differentiability holds.

Asplund's work initiated an extensive line of research that has profoundly influenced the development of modern nonlinear analysis and Banach space theory.

Since then, differentiability theory has evolved in many directions and has been investigated in diverse areas of research far beyond Banach spaces.
Nevertheless, as we observed in our recent paper \cite{KL3}, a significant part of this development has not always been adequately reflected in the existing literature. In particular, several contributions extending classical differentiability theorems to a wider framework of topological vector spaces have received only limited attention.

Motivated by this observation, in this survey article, we continue our effort to provide a more coherent picture of the subject.
The present paper culminates in a comprehensive extension of the celebrated Namioka--Phelps theorem to spaces of continuous functions $C_k(X)$ over arbitrary Tychonoff spaces $X$, thereby considerably enlarging the scope of the classical theory.

In this paper all topological spaces are assumed to be Tychonoff.
They can be described as subspaces of the compact cubes $[0,1]^{T}$.

For a Tychonoff topological space $X$, we denote by
$C_k(X)$ the vector space of all continuous real-valued functions on $X$, equipped with the {\it compact-open topology},
 whose neighbourhood base at zero consists of the sets
$$
V^\ee_K:= \{f\in C_k(X):\ f[K]\subset (-\ee,\ee)\}
$$
where $K$'s run through all nonempty compact subsets of $X$ and $\ee$'s run through all positive numbers; thus $C_k(X)$ becomes a locally convex space.

We use the abbreviation lcs for locally convex spaces.
Let us begin by clarifying the notions of openness, continuity, and the Lipschitz property of mappings in the setting of lcs.

\begin{proposition}\label{ocl}
Let $X$ be an lcs. Assume that $D\subset X$,  $0\in V\subset X$ is a convex open set,  and $\vv: X\to \R$ is a convex function.
The following assertions are equivalent:

(1) $\vv$ is $V$-bounded above at some (every) $x\in D$, i.e., there is a real $t > 0$ such that
 $\sup\vv[(x+tV)\cap D]<\infty$.

(2) $\vv$ is $V$-continuous at some (every) $x\in D$, i.e., for every $\ee>0$ there is a $t>0$ such that $|\vv(x')-\vv(x)|<\ee$ whenever $x'\in (x+tV)\cap D$.
\end{proposition}

\noindent The proof is left to the reader as an exercise.

We will use some basic facts about scattered topological spaces, see for example \cite[Section 8.5]{Semadeni}.
A topological space $X$ is said to be {\it scattered} if every (closed) nonempty
subset $S$ of $X$ has an isolated point in $S$.

Denote by $A^{(1)}$  the set of all non-isolated (in $A$) points of $A \subset X$.
For ordinal numbers $\alpha$, the $\alpha$-th {\it Cantor--Bendixson derivative} of a topological space $X$
is defined by transfinite induction as follows.
\begin{itemize}
\item  $X^{(0)} = X$;
\item  $X^{(\gamma)} = (X^{(\gamma-1)})^{(1)}$ \text{if $\gamma>0$ is a non-limit ordinal};
\item  $X^{(\gamma)} = \bigcap_{\alpha<\gamma} X^{(\alpha)}$ \text{if $\gamma$ is a limit ordinal}.
\end{itemize}
Clearly,  $X$ is scattered if and only if  $X^{(\alpha)}$ is empty for some ordinal $\alpha$.
For a compact space $X$ this implies the following easy assertion.

\begin{lemma} \label{scattered}
A compact space $X$ is scattered if and only if there is no nonempty perfect subset of $X$.
\end{lemma}

The next important results characterize scattered compact spaces in terms of continuous images.

\begin{theorem}\label{sca}\hfill
 \begin{enumerate}
\item {\rm \cite[Proposition 8.5.3]{Semadeni}}
Let $\pi: X \to Y$ be a continuous surjection, where $Y$ is a Hausdorff space.
If $X$ is a scattered compact space then so is $Y$.
\item {\rm \cite[Theorem 8.5.4]{Semadeni}}
A compact space $X$ is scattered if and only if there is no  continuous surjection from $X$ onto the closed interval $[0,1]$.
\end{enumerate}
\end{theorem}

We will also use the following classical theorem of S. Mazurkiewicz and W. Sierpi\'nski.

\begin{theorem}\cite[Theorem 8.6.10]{Semadeni}\label{ordinal}
Every compact scattered first-countable space (in particular, every metrizable compact space) is homeomorphic to a countable compact ordinal.
\end{theorem}

In several places, we  use the following applicable result,  whose proof follows from the classical Tietze--Urysohn extension theorem;
see for example \cite[Extension Lemma, p. 29]{Jarchow} or \cite[Section 3.6]{e}.
\begin{proposition}\label{T}
Let $X$ be a  Tychonoff space and let $L\subset X$ be a compact subset of $X$. Then every continuous function $f: L\rightarrow [0,1]$ can  be extended to a continuous function $F:X\rightarrow [0,1]$  defined on the whole space $X$.
\end{proposition}
Differentiability theory in general topological vector spaces is substantially richer than in Banach spaces. While the notion of G\^ateaux differentiability admits a canonical formulation, several nonequivalent concepts of Fr\'echet differentiability have been developed in the literature. Since both versions will appear naturally in our discussion, we briefly recall the relevant definitions.

Let $Z$ be a topological vector space and let $\vv:Z\to\mathbb{R}$ be a real-valued function defined on a neighbourhood of a point $z\in Z$.

We say that $\vv$ is \emph{G\^ateaux differentiable} at $z$ if there exists a continuous linear functional $\xi:Z\to\mathbb R$ such that
\[
\vv(z+th)-\vv(z)-\xi(th)=o(t)
\]
as $t\to0$ for every $h\in Z$.

The situation is more delicate for Fr\'echet differentiability.

\medskip

\noindent
{\bf (1) Yamamuro Fr\'echet differentiability.}
We say that $\vv$ is \emph{Fr\'echet differentiable at $z$ in the sense of Yamamuro} if $\vv$ is continuous at $z$ and there exists a continuous linear functional $\xi:Z\to\mathbb R$ such that
\[
\sup_{h\in M}
\bigl(\vv(z+th)-\vv(z)-\xi(th)\bigr)=o(t)
\]
as $t\to0$ for every bounded set $M\subset Z$.
This notion was extensively used, among others, in \cite{s,y}.

\medskip

\noindent
{\bf (2) Schwartz Fr\'echet differentiability.}
We say that $\vv$ is \emph{Fr\'echet differentiable at $z$ in the sense of Schwartz} if $\vv$ is continuous at $z$ and there exist a continuous linear functional $\xi:Z\to\mathbb R$ and a neighbourhood $V$ of $0$ such that
\[
\sup_{h\in V}
\bigl(\vv(z+th)-\vv(z)-\xi(th)\bigr)=o(t)
\]
as $t\to0$.
This definition goes back to Schwartz; see, for example, \cite[p.~11]{sc}.

In either case, the functional $\xi$ is uniquely determined, denoted by $\vv'(z)$, and called the \emph{Fr\'echet derivative} of $\vv$ at $z$.

For a detailed overview  of differentiability in topological vector spaces, we refer the reader to the survey of V.I. Averbukh and O.G. Smolyanov \cite{as}.

At first sight, the Schwartz notion appears substantially stronger than the Yamamuro one, since uniformity is required on a neighbourhood rather than merely on bounded sets. In general, the two concepts need not coincide. Remarkably, however, they are equivalent in the principal situations considered in this paper. This happens, first, for the spaces $C_k(X)$, where $X$ is an arbitrary Tychonoff space (see Proposition~\ref{SY}), and, second, for the weak and weak$^{*}$ topologies on Banach spaces, namely when $Z=(E,w)$ or $Z=(E^{*},w^{*})$ for a Banach space $E$; see \cite{s}. As will become apparent throughout the paper, the convexity of the function $\vv$ is the crucial ingredient behind these equivalences.

\begin{proposition}\label{SY} Let $X$ be a Tychonoff topological space, let $D\subset C_k(X)$ be a convex open set, let $\vv: D\to\R$ be a convex real-valued continuous function, and let $f\in D$ be given.

\noindent Then $\vv$ is Schwartz--Fr\'echet differentiable at $f$ if and only if $\vv$ is Yamamuro--Fr\'echet differentiable at $f$.
\end{proposition}
\begin{proof}
 Put
\begin{equation}\label{M}
M:=  \{h\in C_k(X):\ h[X]\subset (-1,1)\}.
\end{equation}
Clearly, $M$ is a bounded set. Using Proposition~\ref{ocl}, we find a compact set $K\subset X$ and a $\gamma>0$
such that
$$
f+V^\gamma_K\subset D\quad {\rm and}\quad \sup\vv [f+V^\gamma_K]<\infty,
$$
where
$$
V^\gamma_K:= \{h\in C_k(X):\ h[K] \subset (-\gamma,\gamma)\};
$$
recall that this $V^\gamma_K$ is a convex open neighbourhood of $0\in C_k(X)$. (This $V^\gamma_K$ will serve as a witness to the
Schwartz--Fr\'echet differentiability of $\vv$ at $f$).

As $\vv$ is Yamamuro--Fr\'echet differentiable at $f$, we have that
$$
\sup_{h\in M}\big(\vv(f+th)-\vv(f)-\vv'(f)(th)\big) = o(t)\quad{\rm as}\quad t\to0.
$$
Hence
$$
\sup_{h\in M}\sum\vv(f\pm th)-2\vv(f) = o(t)\quad{\rm as}\quad t\downarrow0.
$$
Now, consider any $h\in V^\gamma_K$. Then $h[K]\subset (-\gamma,\gamma)$.
Since $X$ is a Tychonoff space and $K\subset X$ is compact, Proposition~\ref{T} yields a continuous
extension $ \check h: X\to (-\gamma,\gamma)$ such that $\check h|_K=h|_K$.
By the Claim below, we get that $\vv(f+th)=\vv(f+t\check h)$, and therefore,
$$
\sup_{h\in V^\gamma_K}\sum\vv(f\pm th)-2\vv(f) = o(t)\quad{\rm as}\quad t\downarrow0.
$$
We also know that $\lim_{t\to0}\frac1t(\vv(f+th)-\vv(f))=\vv'(f)h$ for every $h\in C_k(X)$.
Thus, using the convexity of $\vv$, we can conclude that

$$\sup_{h\in V^\gamma_K}\big(\vv(f+th)-\vv(f)-\vv'(f)(th)\big) = o(t)\quad{\rm as}\quad t\to0.
$$
We proved that $\vv$ is Schwartz--Fr\'echet differentiable at the point $f$, where the neighbourhood $V^\gamma_K$ is a witness for that.

It remains to formulate and prove the announced claim/observation, going back to B. Sharp \cite[p. 208, 209]{s}:
\medskip

\noindent {\tt Claim.}  {\it If $f,f'\in D$ and $f|_K=f'|_K$, then $\vv(f)=\vv(f')$.}
\smallskip

\noindent Proof of Claim.
For all $\lambda>1$ we have $f+\lambda(f'-f)\in D  $ and
$$
f'=\frac{\lambda-1}\lambda f +\frac1\lambda(f+\lambda(f'-f));
$$
hence, by the convexity of $\vv$, we have
\begin{eqnarray*}
\vv(f') &\le& \frac{\lambda-1}\lambda \vv(f) +\frac1\lambda\vv(f+\lambda(f'-f)) \\
        &\le& \frac{\lambda-1}\lambda \vv(f) +\frac1\lambda\sup\vv[f+V_K^\gamma] \longrightarrow \vv(f)\ \ {\rm as}\ \ \lambda\uparrow\infty;
\end{eqnarray*}
Hence $\vv(f') \le \vv(f)$ and, by symmetry, $\vv(f) \le \vv(f')$.
\end{proof}

\noindent We remark that the Yamamuro--Fr\'echet differentiability in the argument above could be replaced by the (yet weaker) `$M$--Fr\'echet differentiability'
where $M$ is defined in (\ref{M}) above.
\smallskip

\begin{remark}
In view of Proposition \ref{SY}, since the paper deals with the $C_k(X)$-spaces, we can and will use only the term
{\it Fr\'echet differentiability} without any extra adjectives.
\end{remark}

Throughout the paper we follow the definitions proposed in \cite{es} and \cite{s}.

\begin{definition}\label{def3}
An lcs $E$ is called an {\it Asplund (weak Asplund) space} if every continuous convex function $f: D \to \R$, where $D\subset E$ is a nonempty open and convex set,
is Fr\'echet (G\^ ateaux, respectively) differentiable at the points of a dense $G_{\delta}$ subset of $D$.
Asplund (weak Asplund) spaces are abbreviated by {\it ASP (WASP}, respectively).
\end{definition}

\begin{definition}\label{def4}
An lcs $E$ is called a {\it Fr\'echet (G\^ ateaux) Differentiability Space} if every continuous convex function $f: D \to \R$, where  $D\subset E$ is a nonempty open and convex set,
is Fr\'echet (G\^ ateaux, respectively) differentiable at the points of a dense subset of $D$.
Fr\'echet (G\^ ateaux) differentiable spaces are abbreviated by {\it FDS {\rm (}GDS}, respectively).
\end{definition}

\begin{remark}\label{notD}
A natural question arises whether in Definitions \ref{def3} and \ref{def4} one can take a subcollection of more specific functions $f$ and obtain the same classes of lcs $E$.
In particular, one can ask: Is it sufficient to consider only a subcollection of those continuous convex functions $f$ whose domains coincide with the whole $E$?

It has been observed in our paper \cite{KL3} that indeed, an lcs $E$ is FDS (GDS) if every continuous convex function $f: E \to \R$
is Fr\'echet (G\^ ateaux, respectively) differentiable at the points of a dense subset of $E$.

Similarly, assuming that $E$ is a Baire lcs, the following holds:
\begin{enumerate}
\item[{\rm (1)}] $E$ is ASP if and only if every continuous convex function $f: E \to \R$ is Fr\'echet differentiable on a dense $G_{\delta}$ subset of $E$.
\item[{\rm (2)}] $E$ is WASP if and only if every continuous convex function $f: E \to \R$ is G\^ ateaux differentiable on a dense $G_{\delta}$ subset of $E$.
\end{enumerate}
For the details we refer the reader to \cite[Propositions 1.7 and 1.9 ]{KL3}.
\end{remark}

The next important statement, due to J. Rainwater, can be found in \cite[Proposition 1.25]{p} and \cite[Lemma 8.23]{y1}.

\begin{proposition}\label{Phelps}
For any Banach space $X$ and every open set $D\subset X$, the (possibly empty) set $G$ of points of  Fr\'echet differentiability of any continuous convex function from $D$ into $\R$ is a $G_{\delta}$.
\end{proposition}
\begin{proof}
Let $f$ be a continuous and convex function defined on the set $D$. Let $S_{X}$ be the unit sphere  of $X$.  For each $n\in\mathbb{N}$ define the following set:
$$G_n=\{x\in D: \inf_{\delta > 0} \sup_{y\in S_{X}} \frac{f(x+\delta y)+f(x-\delta y)-2f(x)}{\delta} < \frac{1}{n}\}.$$
The function $f$ is convex, so it is easy to see that $$\frac{f(x+\delta y)+f(x-\delta y)-2f(x)}{\delta}$$ is decreasing  as $\delta\searrow 0^{+}$.
 Consequently, one gets that $G=\bigcap_{n\in\mathbb{N}} G_n$. Therefore it is enough to show that each set $G_n$ is open.

Fix arbitrary $x\in G_n$.  Since every  continuous convex function is locally  Lipschitz, we can find $\delta_1> 0$ and $M >0$ such that for all $u,v\in B(x,\delta_1)$ one has
$$|f(u)-f(v)|\leq M\|u-v\|.$$

Next, there exist $\delta >0$ and $r >0$ such  that $x\pm \delta y\in D$  and
$$\frac{f(x+\delta y)+f(x-\delta y)-2f(x)}{\delta}\leq r < \frac{1}{n},$$
for all $y$ with $\|y\|=1$.

Take $0<\delta_2< min\{\delta_1, \delta\}$  such that $z\pm \delta_2y\in D$ with $\|y\|=1$ and $r+4M\delta^{-1}\delta_2 < \frac{1}{n}.$  We prove that $B(x,\delta_2)\subset G_n$.
Take $z\in B(x,\delta_2)\subset D.$
Then we get
$$\delta^{-1}|f(z+\delta y))+f(z-\delta y)-2f(z)|\leq$$
$$\delta^{-1}|f(x+\delta y)+f(x-\delta y)-2f(x)| + \delta^{-1}|f(z+\delta y)-f(x+\delta y)|+$$
$$\delta^{-1}|f(z-\delta y)-f(x-\delta y)|+ 2\delta^{-1}|f(z)-f(x)|\leq $$
$$r+4M\delta^{-1}\|z-x\|\leq r+4M\delta^{-1}\delta_2< \frac{1}{n}.$$
 The proof is completed.
\end{proof}
\noindent Consequently, the classes of Banach ASP and FDS spaces coincide.

The classification of locally convex spaces according to the dense and generic differentiability properties of continuous convex functions may be viewed as a natural extension of the pioneering work of E. Asplund \cite{Asplund}, D.G. Larman and R.R. Phelps \cite{Larman}, and I. Namioka and R.R. Phelps \cite{np}, where the theory was developed exclusively in the setting of Banach spaces.

However, several phenomena demonstrate that the relationship between dense and generic differentiability is considerably more subtle than one might initially expect. For instance, M. Talagrand \cite{Tal1} constructed a $1$-Lipschitz function on a Banach space $C(K)$, for a suitable compact space $K$, whose set of G\^ateaux differentiability points is dense but of first category.

Furthermore, M. \v{C}oban and P. Kenderov \cite{Coban} showed that even when the set of G\^ateaux differentiability points of the supremum norm on a space $C(K)$ is dense, it may fail to contain a dense $G_\delta$ subset. A notable example is provided by the double-arrow compact space. One should also recall the intricate example \cite{Moors} of a Banach space that is GDS but not weak Asplund.

A systematic study of locally convex spaces belonging to the classes ASP, WASP, FDS, and GDS was initiated by B. Sharp in 1990 \cite{s}. This line of research was subsequently developed in a series of joint papers \cite{es, Eyland_Sharp2, Eyland_Sharp3}, which established many of the basic properties of these classes. Unfortunately, these significant contributions have not always been adequately reflected in later developments of the subject. In our recent paper \cite{KL3}, we revisited this circle of ideas, clarified several historical and bibliographical issues, and investigated the preservation of the WASP and ASP properties under arbitrary products of locally convex spaces.

We claim no novelty for the statements in this survey article either.
Our aim is to present all major results with self-contained complete proofs, making the paper accessible to both specialists and newcomers.

\section{Namioka--Phelps theorem and additional supplementing results}

Assume first that Fr\'echet differentiability points of continuous convex functions can be detected on separable Banach spaces, that is, on Banach spaces possessing a countable dense subset.

A natural question then arises: can such ``separable'' information be lifted to an arbitrary Banach space? More precisely, if a continuous convex function $f$ is defined on a Banach space $E$, can one infer that $f$ admits at least one point of Fr\'echet differentiability from the corresponding result established on suitable separable subspaces of $E$?

The answer is affirmative and relies on the existence of appropriate separable reduction techniques for Fr\'echet differentiability.

Our next auxiliary result is motivated by the property formulated in Proposition~\ref{Phelps}.
For the notions appearing in Proposition \ref{111} below, we refer the reader to \cite[p.~317--321]{y1} and \cite[Chapter~11, pp.~625--629]{y2}.
In particular, if $f$ is concave and the Fr\'echet superdifferential $\partial_F f(x)$ is nonempty, then $f$ is Fr\'echet differentiable at $x$ and
$\partial_Ff(x)=\{f'(x)\}$.

\begin{proposition}\label{111} Let $(E,\nn)$ be a (separable) Banach space whose dual $E^*$ is separable,
let $D\subset E$ be an open set, and let $f:D\longrightarrow(-\infty,\infty]$ be a lower semi-continuous function.

\noindent Then the set of all $x\in D$, where the Fr\'echet subdifferential $\partial_F f(x)$
is nonempty, is dense in $D$.
\end{proposition}

\begin{proof} We first realize that $E$ admits an equivalent Fr\'echet smooth (off the origin) norm. Indeed,
let $\{x_1, x_2,\ldots\}$ be a countable dense subset of the unit ball $S_E$ of $E$ and $\{\xi_1, \xi_2, \ldots\}$ be a countable dense subset of the unit ball $S_{E^*}$ of the dual $E^*$. Then the assignment sending $x^*\in E^*$ to
$$
|x^*|=\sqrt{\|x^*\|^2 + \hbox{$\sum_{n=1}^\infty 2^{-n}$}\langle x^*,x_n\rangle^2+\hbox{$\sum_{n=1}^\infty 2^{-n}$}{\rm dist}\, (x^*,{\rm sp}\{\xi_1,\ldots,\xi_n\}\big)^2}
$$
is an equivalent weak$^*$ lower semi-continuous norm on $E^*$. For completeness, we provide the proof of the subadditivity of $|\cdot|$, i.e., why the inequality $$|x^*+y^*| \le |x^*|+|y^*|$$ holds,
as explicit verifications are rarely detailed in standard literature.

A possible argument may exploit the subadditivity of the famous (Euclidean) $\ell_2$-norm.

A `small' Kadec--Klee like effort, see \cite[p. 113-117]{d}, reveals that
the norm $|\cdot|$ is locally uniformly rotund, i.e., if $x^*, x^*_1,\ x^*_2,\ldots$ are elements of $E^*$ such that
$$2|x^*|^2+2|x_i^*|^2-|x^*+x^*_i|^2\to0$$ when $i\to\infty$, then $|x^*_i-x^*|\to 0$ when $i\to\infty$.
Assume this is not so. Then, find an $\ee>0$ and an infinite set $N\subset \N$ such that $\|x^*_i-x^*\| > 3\ee$
for every $i\in N$.
Find $m\in N$ so big that dist$(x^*,{\rm sp}\{\xi_1,\ldots,\xi_m\}\big)<\ee$.

A simple convexity argument provides an infinite set $N_1\subset N$ such that
dist$(x^*_i,{\rm sp}\{\xi_1,\ldots,\xi_m\}\big)<\ee$ for every $i\in N_1$. For every such $i$  find
$$y^*_i\in {\rm sp}\{\xi_1,\ldots,\xi_m\}$$ such that $\|x^*_i-y^*_i\|<\ee$, and so,
$$
\|y^*_i-x^*\| \ge \|x^*_i-x^*\| - \|x^*_i-y^*_i\| > 3\ee -\ee=2\ee.
$$

The sequence $(y^*_i)_{i\in N_1}$ is bounded
because $\|x^*_i\|\to\|x^*\|$ as $i\to\infty$, by convexity. Further, since it lies in a finite-dimensional space, there are $y^*\in {\rm sp}\{\xi_1,\ldots,\xi_m\}$ and an infinite set
$N_2\subset N_1$ such that $\|y^*_i-y^*\|\to0$ as $N_2\ni i\to\infty$.
Hence $\|y^*-x^*\|\ge 2\ee.$

Find $n\in\N$ so that $$\|\langle y^*-x^*,x_n\rangle\| >\ee.$$
Another  simple convexity argument yields that $$\langle x^*_i,x_{n}\rangle \to \langle x^*,x_{n}\rangle$$ as $i\to\infty.$
Hence we have that
$$
(\ee>)\ \| y^*_i-x^*_i\| \ge  \langle y^*_i-x^*_i, x_{n}\rangle> \ee
$$
for all $i\in N_2$ big enough, a contradiction. We proved that our norm $|\cdot|$ on $E^*$ is LUR.

Now, having proved that our norm $|\cdot|$ on $E^*$ is LUR, the \v Smulyan test guarantees that the predual norm $|\cdot|$ on $E$, defined by
$$
E\ni x\longmapsto\sup\big\{\langle x^*,x\rangle:\ x^*\in E^*,\ |x^*|\le1\big\}=:|x|,
$$
is Fr\'echet differentiable at every nonzero point of $E$. For more details, see \cite[p. 43]{dgz}.

Next, let $\overline x\in Df$ and $\ee>0$ be given. We will find a $v\in E$ such that
$|v-\overline x|<\ee$ and $\partial_F f(v)$ is nonempty.

From the lower semi-continuity of $f$ find $\ee'\in(0,\ee)$ so small that $B(\overline x,\ee')\subset D$
and that $f$ is bounded below on the open ball
$B(\overline x,\ee')$. Define
\begin{eqnarray*}
\vv(x):=\begin{cases}
          \big(\tan(\hbox{$\frac\pi{2\ee'}$} |x-\overline x|)\big)^2 & \text{if $x\in B(\xx,\ee')$}\\
          \infty & \text{if $x\in E\setminus B(\xx,\ee')$}.
        \end{cases}
\end{eqnarray*}
Then $\vv:E\longrightarrow[0,\infty]$ is easily seen to be proper and lower semi-continuous.

Now, the Borwein--Preiss variational principle \cite[Theorem 4.20]{p} provides a Fr\'echet differentiable function
$\theta: D\longrightarrow[0,\infty)$ such that the sum $f+\vv+\theta$ attains infimum at some $v\in D$; clearly,
$v\in B(\xx,\ee')$. We thus have
$$
f(v+h)+\vv(v+h)+\theta(v+h) \ge f(v)+\vv(v)+\theta(v)\quad{\rm for\ every}\quad h\in E,
$$
and so
$$
f(v+h)-f(v)+\langle \vv'(v)+\theta'(v),h\rangle \ge -o(|h|)\quad{\rm as}\ \ h\in E\ \ {\rm and}\  \ |h|\to0.
$$
We proved that $-\vv'(v)-\theta'(v)\in \partial_F f(v)$, and hence $\partial_F f(v)\neq\emptyset$.
\end{proof}
\begin{remark}
The proof above may be viewed as a refinement of an earlier approach based solely on Ekeland's variational principle. Indeed, before the advent of the Borwein--Preiss principle, essentially the same conclusion could be obtained by using Ekeland's principle alone.
\end{remark}
Having established Proposition~\ref{111} in the separable case, one is naturally led to ask whether the separability assumption is really necessary. Remarkably, it is not. The proposition admits an extension to arbitrary Asplund spaces, and the roots of this result can be traced back to Gregory's work \cite[Theorem~2.14]{p}.

The contrast with G\^ateaux differentiability is striking. While the Fr\'echet theory extends successfully to the nonseparable setting, an analogous statement for G\^ateaux differentiability is simply false. A classical obstruction is provided by the limsup norm on $\ell_\infty$, which serves as a counterexample.
\begin{proposition}\label{g}
Let $(E,\nn)$ be a non-separable Banach space, $f:D\rightarrow\R$ be a  convex continuous function
defined on an open convex set $D\subset E$, and $Z$ be a separable subspace of $E$.

\noindent Then there exists a separable subspace $Y$ of $E$, containing $Z$, with $D\cap Y$ nonempty, and such that, if
the restriction $f|_Y$ of $f$ to $Y$ is Fr\'echet differentiable at some $x\in D\cap Y$, then  $f$ is
Fr\'echet differentiable at $x$.
\end{proposition}

\begin{proof}
First, we need a  translation
of Fr\'echet differentiability (of convex functions)  to the terms of the space $E$:
\sl $f$ is Fr\'echet differentiable at $x\in E$ if and only if
$$
S(x,t):=\sup_{h\in B_X}\big(f(x+th)+f(x-th)\big)=2f(x)+o(t)\ \ {\rm as}\ \ \ Q_+\ni t\downarrow 0;
$$
\rm this is easy to check. (Here and below, $Q_+$ means the set of all positive rational numbers). 
For any $x\in E$ and any $t>0$, if $S(x,t)<\infty$, we find a vector $u(x,t)\in B_E$ such that
\begin{eqnarray}\label{22}
f\big(x+ t\, u(x,t)\big) +f\big(x-t\,u(x,t)\big) > S(x,t)-t^2.
\end{eqnarray}
(This $u(x,t)$ is almost ``the worst possible'' regarding the
Fr\'echet differentiability of $f$ at $x$).
Let $C_0$ be a countable dense subset of $Z$.

We  construct countable sets
$$C_0\subset C_1\subset C_2\subset\cdots\subset E$$ as follows. Let $m\in\mathbb{N}$ be given and assume
that $C_{m-1}$ was already found.

Find a countable set $C_m$ in $E$ that
 is stable under taking all finite linear combinations with rational coefficients, and such that
it contains $C_{m-1}$ as well as the set $$\big\{u(x,t):\ x\in C_{m-1},\ t\in Q_{+}\};$$
clearly, $C_m$ is again countable.
Do so for every $m\in\mathbb{N}$, and put finally $Y:=\overline{C_1\cup C_2\cup\cdots}\,$.
Clearly, $Y\supset Z$.

We claim that this $Y$ has the desired property. So, assume that
$f|_Y$ is Fr\'echet differentiable at some $x\in Y$. We show that the function $f$ itself is
Fr\'echet differentiable at $x$ as well. (If $x\in C_1\cup C_2\cup\cdots$, then the argument is straightforward.
So, we have to find an argument working also for $x\in Y\setminus C_1\cup C_2\cup\cdots$).

Let $L$ denote a Lipschitz constant of $f$ in a vicinity of $x$; see \cite[Proposition 1.6]{p}.
Pick any $t\in Q_+$ small enough (so that we can exploit the $L$-Lipschitz property of $f$ around $x$).
Then, find $c\in\bigcup_{m=1}^\infty C_m$ such that $\|c-x\|<t^2$. We can now
subsequently estimate for all sufficiently small $t\in Q_{+}$ (so that we can use the
Lipschitz property of $f$).
\begin{eqnarray*}
2f(x)&\le& S(x,t) < S(c,t)+ 2Lt^{2}\\
&<& f\big(c+ t\,  u(c,t)\big)+f(c-t\, u(c,t)\big)+t^2 + 2Lt^{2}\\
&<& f\big(x+ t\,  u(c,t)\big)+f(x-t\, u(c,t)\big)+t^2 + 4Lt^{2}\\
&\le&   \sup_{k\in B_Y}\big(f(x+ t  k)+f(x-tk)\big)+t^{2} + 4 Lt^{2}\\
&=&o(t) + 2f(x)\quad {\rm as}\quad Q_{+}\!\backepsilon t\downarrow0
\end{eqnarray*}
since $f|_Y$ is Fr\'echet differentiable at $x$.
Therefore, the ``whole'' $f$ is Fr\'echet differentiable at $x$.
\end{proof}
The proof of our main result, Theorem~\ref{MAIN}, relies on several auxiliary results of independent interest. Among them, a central role is played by the celebrated theorem of I. Namioka and R.R. Phelps, a cornerstone of modern differentiability theory; see \cite[Lemma~VI.8.3]{dgz}, \cite[Theorem~14.25]{y2}, or \cite[Theorem~18]{np}. For a broader perspective on this circle of ideas, we also recommend the monograph \cite{f}. Another very useful relevant source of information is the survey article \cite{Yost}.

The proof of Theorem~\ref{1} borrows certain topological arguments from \cite[pp.~626--627]{y2}; see also \cite[Main Theorem]{Pelczynski-Semadeni}. It is worth emphasizing, however, that our approach is entirely independent of differentiability techniques and relies solely on topological methods.
\begin{theorem}\label{1} For a compact space $X$ the following assertions are equivalent:
 \begin{enumerate}
  \item $X$ is scattered.
  \item Every separable subspace of $C(X)$ has a separable dual.
  \item The Banach space $C(X)$ does not contain a copy of $\ell_1$ isomorphically.
 \end{enumerate}
\end{theorem}

\begin{proof}
(1)$\Longrightarrow$(2) Let $Y\subset C(X)$ be any separable subspace. Choose a countable and linearly independent set $N\subset B_Y$ such that its linear span is dense in $Y$.

Consider a correspondence
$$
X\ni x\longmapsto \{n(x):\ n\in N\}=:\psi(x)\in [-1,1]^{N}.
$$
Thus $\psi: X\longrightarrow [-1,1]^{N}$, and this is a continuous mapping. Denote
$L:=\psi(X)$; this is a metrizable compact space.
Since $X$ is scattered,  $L$ is scattered as well, by Theorem \ref{sca}. Moreover, $L$ is countable, by Theorem \ref{ordinal}.

Next, let us enumerate $L$ as $\{z_{n}:\ n\in N\}$.
It is known that every  linear continuous functional $x^*\in C(L)^{*}$ has the form
$$
x^*(f)=\sum_{n}a_n f(z_n),\quad f\in C(L),
$$
and $\|x^*\|=\sum_{n\in N} |a_n| < \infty$,
where $(a_n)_{n\in N}$ is a suitable element of $\ell_1(L)$, see \cite[Main Theorem (11)]{Pelczynski-Semadeni} or \cite[Corollary 19.7.7]{Semadeni}.
Consequently, $C(L)^{*}$ is isometric to $\ell_1(L)$, and this implies that $C(L)^*$ is separable.

We will show that the subspace $Y\subset C(X)$ is isometric to a subspace of $C(L)$.  
For $n\in N$ and $x\in X$ we put
$$
\tilde n(\psi(x)):= n(x).
$$
The $\tilde n$ is a well defined function on $L$. Indeed, if $\psi(x)=\vv(h)$ for some $x, h\in X$, then  $\tilde n(\psi(x))= n(x)=n(h)=\tilde n(\psi(h))$.

Fix any $n\in N$. The function $\tilde n$ belongs to $C(L)$. Indeed, consider a net $(\psi(x_\alpha))$ in $L$ converging to some $\psi(x)\in L$.
This means that $m(x_\alpha)\to m(x)$ for every $m\in N$. Thus, in particular, $n(x_\alpha)\to n(h)$. Therefore, $$\tilde n(\psi(x_\alpha))=n(x_\alpha)\longrightarrow n(x)= \tilde n(\psi(x)).$$

Now, for any element $(a_n)_{n\in N} \in c_{00}(N)$ we have
\begin{eqnarray*}
\Big\|\sum_{n\in N} a_n \tilde n\Big\|&=& \sup_{l\in L}\Big|\sum_{n\in N} a_n\tilde n (l)\Big| =
\sup_{x\in X}\Big|\sum_{n\in N} a_n\tilde n (\psi(x))\Big| \cr
&=& \sup_{x\in X}\Big|\sum_{n\in N} a_nn(x)\Big|= \Big\|\sum_{n\in N}a_nn\Big\|.
\end{eqnarray*}
(In both $C(L)$ and $C(X)$, we considered the maximum norm $\|\cdot\|$).
Hence the mapping
$$
{\rm sp}\, N\ni \sum_{n\in N}a_nn\longmapsto \sum_{n\in N} a_n\tilde n \in C(L)
$$
is well defined, and it is a linear isometry.  Also, we know that the linear span sp$(N)$ is dense in $Y$. Thus, we can uniquely extend this mapping to the
linear isometry defined on the whole $Y$ into $C(L)$, say $T: Y\hookrightarrow C(L)$. 

Finally, since the operator $T$ is injective, a simple consequence of the Hahn--Banach theorem reveals that the adjoint operator $T^*: C(L)^*\to Y^*$ is
surjective. Recalling that $C(L)^*$ was separable, $Y^*$ must be separable as well by the Hahn-Banach Theorem.

\smallskip
(2)$\Longrightarrow$(3) Is clear since the dual of $\ell_1$ is isometric to the non-separable space $\ell_\infty\,$.

\smallskip
(3)$\Longrightarrow$(1) Assume on the contrary that $X$ is not scattered. By Theorem \ref{sca}, there is a continuous surjection $\rho: X\twoheadrightarrow [0,1]$.
Then the adjoint mapping $\rho^{*}: C[0,1]\hookrightarrow C(X)$ is a linear isometry into. However, the Banach space $C[0,1]$ is universal for the class of all separable Banach spaces by the
Banach--Mazur theorem, see \cite[Proposition 1.5]{Pelczynski-Bessaga}. Therefore, $C(X)$ contains an isometric copy of $\ell_1$. This contradiction completes the proof.
\end{proof}
A crucial ingredient in the proof of Theorem~\ref{MAIN} is the following celebrated result of I. Namioka and R.R. Phelps \cite{np}, one of the cornerstones of the differentiability theory of Banach spaces $C(X)$. Since it plays a central role in our approach, we present a complete proof.
\begin{theorem}[Namioka--Phelps]\label{Namioka+Phelps} For every compact space $X$,
the Banach space $C(X)$ is Asplund if and only if $X$ is scattered.
\end{theorem}
\begin{proof}
The necessity follows from the more general Theorem~\ref{MAIN}. Indeed, it suffices to replace $X$ by $L$ in the proof of the implication
$(1)\Rightarrow(2)$ given there.

To prove the converse, assume that $X$ is scattered. Let $D\subset C(X)$ be an open convex set and let
$\vv:D\to\mathbb{R}$ be a continuous convex function. We shall show that $\vv$ is Fr\'echet differentiable on a dense subset of $D$.

Let $U$ be an arbitrary nonempty open subset of $D$. Choose a separable subspace
$Z\subset C(X)$ such that $U\cap Z\neq\varnothing$. Applying Proposition~\ref{g}, we obtain a separable superspace
\[
Z\subset Y\subset C(X)
\]
with the properties described there.

Since $X$ is scattered, Theorem~\ref{1} yields that the dual space $Y^{*}$ is separable. It follows from Proposition~\ref{111} that $Y$ is an Asplund space. Therefore, the restriction
$\vv|_{Y}$ is Fr\'echet differentiable at some point
\[
f\in U\cap Y,
\]
because $U\cap Y$ is a nonempty open subset of $Y$.

Finally, Proposition~\ref{g} ensures that the Fr\'echet differentiability of $\vv|_{Y}$ at $f$ implies the Fr\'echet differentiability of the original function $\vv$ at the same point. Thus every nonempty open subset of $D$ contains a point at which $\vv$ is Fr\'echet differentiable. Consequently, the set of Fr\'echet differentiability points of $\vv$ is dense in $D$.

The proof is completed by an application of Proposition~\ref{Phelps}.
\end{proof}
\section{Asplund spaces $C_k(X)$ over Tychonoff spaces $X$}
Before proving Theorem~\ref{MAIN}, we shall need the following result. A different proof can be found in \cite[Proposition~3.13]{KL1}.

For convenience, we also fix some terminology. Let $X$ be a topological vector space, $Y$ a Banach space, and $T:Y\to X$ a linear mapping. The operator $T$ is called an {\it embedding} (or an {\it isomorphism into} $X$) whenever $T$ is injective and continuous, and the inverse mapping
$T^{-1}:T[Y]\to Y$ is continuous. Equivalently, $T$ induces a topological isomorphism of $Y$ onto its image $T[Y]$.

Throughout the paper, $\stackrel{\circ}{B}_Y$ denotes the open unit ball of $Y$.
\begin{lemma}\label{ku} Let $X$ be a Tychonoff space and assume that there is a Banach space $Y$ which is isomorphic to a subspace of $C_k(X)$.

\noindent Then there exists a compact set $K\subset X$ such that $Y$ is isomorphic to a subspace of the (Banach) space $C(K)$.
\end{lemma}

\begin{proof}
Let $S:Y\hookrightarrow C_k(X)$ be the promised isomorphism into. Then there must exist some $\ee>0$, $\alpha>0$, and a compact subset $K\subset X$ such that
\begin{equation}\label{14}
V^\ee_K \cap S[Y] \subset S[\stackrel \circ B_Y]\subset V^{\ee/\alpha}_K\,.
\end{equation}

We show that the (Banach) space $Y$ is isomorphic to a subspace of the (Banach) space $C(K)$.
Define $\tilde S:Y\to C(K)$ by
$$
\tilde S(y):= (Sy)|_K, \quad y\in Y.
$$
$\tilde S$ is (linear and) injective. Indeed, take any $y\in Y$ such that $\tilde Sy=0$. This means that
$(Sy)|_K\equiv0$. Then, for every $n\in \N$ we have
$$
n Sy\in V^\ee_K\cap S[Y] \subset S[\stackrel \circ B_Y]
$$
by (\ref{14}), and so $Sy\in S[\frac1{n} \stackrel \circ B_Y]$, and hence $y\in \frac 1{n} \stackrel \circ B_Y$, as $S$ is injective. Therefore, $y=0$.

Now, fix any $y\in \stackrel \circ B_Y$. By (\ref{14}) $Sy \in V^{\ee/\alpha}_K$, i.e., $Sy[K] \subset (-\ee/\alpha,\ee/\alpha)$, and hence $\|\tilde Sy\| <\ee/\alpha$. We thus proved that
$$
\tilde S[\stackrel \circ B_Y] \subset \frac\ee\alpha \stackrel \circ B_{C(K)}.
$$
It remains to prove that
$\ee \stackrel \circ B_{C(K)}\cap \tilde S[Y] \subset \tilde S[\stackrel \circ B_Y]$.
Fix any $y\in Y$ such that $$\tilde Sy \in \varepsilon \stackrel \circ B_{C(K)};$$ thus $\|\tilde Sy\| <\varepsilon$.
Hence 
$Sy\in V^\ee_K \cap S[Y]$.
By (\ref{14}), there is a $y'\in \stackrel\circ B_Y$ so that $Sy=Sy'$. But $S$ is  injective. Hence $ y=y'$ and so $y\in \stackrel \circ B_Y$.
We proved that
$$\varepsilon \stackrel \circ B_{C(K)} \cap \tilde S[Y] \subset \tilde S[\stackrel \circ B_Y],$$
and therefore, $\tilde S$ is an isomorphism of $Y$ onto $\tilde S[Y] \  \ (\subset C(K)$). Summarizing, we obtained that
$$
\ee \stackrel \circ B_{C(K)}\cap \tilde S[Y] \subset \tilde S[\stackrel \circ B_Y] \subset \frac\ee\alpha \stackrel \circ B_{C(K)};
$$
that is, analytically (using the injectivity of $\tilde S$),
$$
\forall y\in Y\quad \ee\|y\|_Y \le  \|\tilde Sy\|_{C(K)} \le \frac\ee\alpha \|y\|_Y.$$
  \end{proof}

Following M. Komisarchik and M. Megrelishvili \cite{KM}, a locally convex space $E$ is said to have the \emph{Namioka--Phelps property} if every bounded subset $B\subset E$ is fragmented on each weak$^*$-compact equicontinuous subset $K\subset E^*$. Explicitly, this means that for every nonempty subset $A\subset K$ and every $\varepsilon>0$, there exists a weak$^*$-open set $U\subset E^*$ such that $U\cap A\neq\emptyset$ and
\[
\operatorname{diam}\{\langle v,x\rangle : x\in U\cap A,\ v\in B\}<\varepsilon.
\]

The equivalence $(1)\Leftrightarrow(2)$ in Theorem~\ref{MAIN} below was established in \cite[Corollary 3.8]{es} as a consequence of several intermediate results. Our proof is more direct and, to a large extent, self-contained.

The equivalences $(2)\Leftrightarrow(3)\Leftrightarrow(4)$ were proved in \cite{GKKM}, while the equivalences $(2)\Leftrightarrow(5)\Leftrightarrow(6)$ were obtained in \cite{KM}.
We shall not address the latter equivalences here, since their proofs rely on a substantial amount of fragmentation theory that lies beyond the scope of this paper. Our main objective is to present a self-contained proof of the classical Namioka--Phelps theorem.
\begin{theorem}\label{MAIN}
Let $X$ be a Tychonoff space. The following assertions are equivalent:
\begin{enumerate}
\item   $C_k(X)$ is an Asplund space.
\item   Every compact set in $X$ is scattered.
\item  The space $C_k(X)$ does not contain any isomorphic copy of $\ell_{1}$.
\item Every separable Banach subspace of $C_k(X)$ has separable dual.
\item $C_k(X)$ satisfies the Namioka--Phelps property.
\item $C_k(X)$ is tame, i.e., every weak$^{*}$-compact equicontinuous convex subset $K\subset C_k(X)^*$ is the closed (in the strong  topology) convex hull of the extreme points of $K$.
\end{enumerate}
\end{theorem}
\begin{proof}
(1)$\Longrightarrow$(2)
  Assume that $C_k(X)$ is an Asplund space and that some compact subset $L\subset X$ is not scattered. Then, find a nonempty closed perfect set $P\subset L$, by
Lemma~\ref{scattered}.
Consider the function $\vv: C_k(X)\to\R$ defined by
$$
C_k(X)\ni f\longmapsto \max f(P)=:\vv(f).
$$
Clearly, $\vv$ is convex. It is also continuous on $C_k(X)$. Indeed, pick any $f\in C_k(x)$ and any $\ee>0$.
Then
for every $h\in V_P^\ee$ we have
$$
\vv(f+h)-\vv(f)\le \max f(P)+\max h(P)-\max f(P) < \ee
$$ and, similarly,
$$\vv(f)-\vv(f+h)\le \max (f+h)(P)+\max (-h)(P)-\max (f+h)(P) < \ee;
$$
thus $|\vv(f+h)-\vv(f)| <\ee$.

Put
$$
M:=\{h\in C_k(X):\ h(X) \subset [-1,1]\};
$$
this is obviously a bounded set.
As $C_k(X)$ is Asplund and $\vv$ is continuous convex,
$\vv$ is Fr\'echet differentiable at some $f\in C_k(X)$.
Thus, in particular, there is a linear continuous $\xi: C_k(X)\rightarrow\R$ such that
$$
\sup_{h\in M}\Big|\frac{(\vv(f+th)-\vv(f)}t - \xi(h)\Big|\longrightarrow 0\quad {\rm as}\quad t\to0;
$$
and hence (getting rid of the almost redundant $\xi$)
$$
(0 \le )\ \sup_{h\in M}\frac{(\vv(f+th)+\vv(f-th)-2\vv(f)}t\longrightarrow 0\quad {\rm as}\quad t\to0;
$$
Thus, in particular for $\ee:=\frac12$ there is a $\delta>0$ such that
$$
\vv(f+\delta h)+\vv(f-\delta h)-2\vv(f) < \frac\delta2
$$
for every $h\in M$.

Now, pick a $p\in P$ so that $f(p)=\vv(f)$. As $P$ is perfect, there is a $q\in P\setminus \{p\}$ such that $f(q)> \vv(f)-\frac\delta 2$.
We recall that $X$ is a Tychonoff space and $L$ is a compact subset of $X$. Hence, there is a continuous function $h: X\to [0,1]$ such that
$h(p)=0$ and $h(q)=1$. Note that then $\vv(h)=1$ and $h\in M$.
Thus
\begin{eqnarray*}
\delta & =& \delta h(q)-\delta h(p) = (f+\delta h)(q) + (f-\delta h)(p) -f(q)-f(p) \cr
& <& \vv(f+\delta h) + \vv (f-\delta h) - 2\vv(f) +\frac\delta 2 < \frac\delta 2 + \frac\delta 2 =\delta
\end{eqnarray*}
(Note that we used a \v Smulyan like argument, see \cite{dgz}) and \cite[p. 342--344]{y2}; a contradiction. We proved that $L$ must be scattered.
 \smallskip

(2)$\Longrightarrow$(1)  The following proof adapts and extends the geometric techniques established in  \cite{s} and \cite{es}. Let $D\subset C_k(X)$ be any (nonempty) open convex set and consider any convex function $\vv: D\to \R$, everywhere continuous in the $k$-topology of the space $C_k(X)$.
We want to show that the set of points where $\vv$ is Fr\'echet differentiable contains a dense $G_\delta$ subset of $D$.

Pick some $u\in D$. The continuity of $\vv$ at $u$ yields an $\ee>0$ and a compact subset $L\subset X$ such that $u+V^\ee_L \subset D$ and that $\vv$ is bounded above on the set $u+V^{\ee}_L$.
A simple argument following [33, pp. 4-5], then reveals that
for every $v\in D$ there is a $\gamma>0$ such that $v+ V^{\gamma}_L \subset D$ and that
$\vv$ is bounded above on $v+ V^{\gamma}_L$.

\smallskip

Let $C(L)$ be the Banach space of all continuous functions on $L$, with `maximum' norm. Its closed unit ball (around the origin) is denoted by $B_{C(L)}$. Consider the 'restriction' mapping $T: C_k(X)\to C(L)$ defined by
$$
C_k(X)\ni f\longmapsto f|_L=:Tf \in C(L).
$$
Clearly, $T$ is linear. It is also continuous since, for every $\gamma>0$, we have $T[V^\gamma_L]\subset
\gamma\!\!\stackrel \circ B_{C(L)}$. (We have here even equality by Proposition~\ref{T}). Since $T$ is continuous, $T^{-1}(\Omega)$ is open or $G_\delta$
whenever $\Omega\subset C(L)$ is, respectively, open or $G_\delta$. Let us further check that $T[D]$ is an open subset of $C(L)$.
So, fix any $f\in D$. Find a $\gamma>0$ so that $f+ V^\gamma_L\subset D$. Then
$$
Tf +\gamma\!\!\stackrel \circ B_{C(L)} = T[f+ V^\gamma_L] \subset T[D];
$$
the equality here comes from Proposition~\ref{T}, as $X$ is Tychonoff and $L$ is compact.
We verified that $T[D]$ is open.

We further observe that, given any dense subset $\Omega$
in $T[D]$, then $T^{-1}[\Omega]$ is dense in $D$. Indeed,
pick any $v\in D$, any $\beta>0$, and any compact set $L'$ in $X$. We want to show that $$(v+V^\beta_{L'})\cap T^{-1}
[\Omega]\neq\emptyset.$$ Take $\gamma\in(0,\beta)$ so small that $v+V^{\gamma}_L\subset D$.
As $\Omega$ is dense in $T[D]$, we have that
$$
(Tv+\gamma\!\! \stackrel \circ B_{C(L)})\cap \Omega\neq\emptyset;
$$
hence, there is an $$h\in \gamma\!\! \stackrel \circ B_{C(L)}$$ so that $Tv+h\in\Omega$.
By Proposition~\ref{T},
we find a continuous function $\check h:X\to\R$ such that
$\check h[X] \subset (-\gamma,\gamma)$ and $\check h|_L=h$. Then $$v+\check h \in v + V^{\gamma}_{L'} \subset v + V^{\beta}_{L'}$$
and $T(v+\check h)=Tv + h\ (\in\Omega)$. Thus $$v+\check h\in (v+V^\beta_{L'})\cap T^{-1}[\Omega].$$
We proved that $T^{-1}[\Omega]$ is dense in $D$; note that $T^{-1}[\Omega]\cap D$ {\sl is dense in} $D$, too.

Now, we define the function $\psi: T[D]\to\R$ as
$$
\psi(g):=\vv(\check g),\quad g\in T[D],
$$
where $\check g$ is {\tt any} element of $D$ 
such that $\check g|_L=g$.
This function is well-defined because, 
if $f, f'\in D$ and $f|_L=f'|_L=g$, then by the Claim from the proof of Proposition~\ref{SY}, $\vv(f)=\vv(f')$.

The function $\psi$ is convex. Indeed, take any $g_1,g_2\in {C(L)}$ and any $\alpha\in (0,1)$.
Find some $\check{g_1}, \check{g_2}\in D$, with $\check{g_1}|_L=g_1$ and $\check{g_2}|_L=g_2$.
Then 
$$\big(\alpha \check{g_1} +(1-\alpha)\check{g_2}\big)|_L= \alpha g_1+(1-\alpha)g_2,$$ and
hence
\begin{eqnarray*}
\psi\big(\alpha g_1+(1-\alpha)g_2\big) &=& \vv \big(\alpha \check{g_1} +(1-\alpha)\check{g_2}\big) \le \alpha \vv(\check{g_1}) + (1-\alpha) \vv(\check{g_2}\big)\cr
&=&\alpha \psi({g_1}) + (1-\alpha) \psi({g_2}).
\end{eqnarray*}

Let us check that our $\psi$ is continuous on $T[D]$. Fix any $f\in D$. We know that there is an $\ee>0$ so that $f+V^\ee_L$ lies in $D$ and that
$\vv[f+ V^\ee_L]$ is bounded above.
Thus the set $$\psi[Tf+\ee\!\!\stackrel \circ B_{C(L)}] = \vv[f+V^\ee_L]$$ is also bounded above (we needed Proposition~\ref{T}). Given this, \cite[p. 4, 5]{p} (valid for Banach spaces) guarantees the continuity of $\psi$.

Now, since $L$ is a scattered compact set by the assumption,  Theorem \ref{Namioka+Phelps}  says that
the Banach space $C(L)$ is Asplund. Hence our $\psi$ is Fr\'echet differentiable at the points of
a dense $G_\delta$ subset, say $\Omega$, of $T[D]$.
Fix any $f$ in $T^{-1}(\Omega)\cap D$ (which is already known to be a dense $G_\delta$ subset of $D$).
Find an $\ee>0$ so small that $f+V^\ee_L\subset D$ (we already proved that such an $\ee$ exists).
As $Tf\in\Omega$, the function $\psi$ is Fr\'echet differentiable at $Tf$, with derivative, $\xi: C(L)\to\R$, say.
Having this we get
\begin{eqnarray*}
&&\sup_{h\in V_L^\ee}\big|\vv(f+ th)-\vv(f) - \xi(th|_L)\big| \cr
&= & \sup_{z\in \ee \stackrel \circ B_{C(L)}}\big( \psi(T f + tz)-\psi(T f)-\xi(tz)\big) = o(t)\,\,{\rm as}\,\,(-\ee,\ee)\ni t\to 0;
\end{eqnarray*}
Here we again used Proposition~\ref{T}. Therefore, our $\vv$ is Fr\'echet differentiable at $f$, with derivative $$C_k(X)\in h \longmapsto \xi(h|_L)=: \vv'(f)h.$$
In addition, we should verify that this $\vv'(f)$ is linear and continuous. Indeed, we have that
$\xi[V^\ee_L] \subset (-\ee,\ee)$.
\smallskip
(3)$\Longrightarrow$(2)
Assume that $X$ contains a compact subset $K$ which is not scattered. By Theorem~\ref{sca}, there exists a continuous surjection $\rho: K\twoheadrightarrow [0,1]$.
Using Proposition~\ref{T}, $\rho$ can be extended to a continuous surjection, say  $\check \rho: X\twoheadrightarrow [0,1]$.
The natural adjoint mapping to $\check\rho$, say $T: C[0,1]\hookrightarrow C_k(X)$, is then an  embedding (i.e., an isomorphism onto its range). Indeed, a bit of routine calculation  reveals that
$$
V^1_K \cap T[C[0,1]]= T\big[\stackrel \circ B_{C[0,1]}\big]\  \ (\subset V^1_L\ \ {\rm for \ every\  compact}\  \ L\subset X).
$$
(We note that the above equality says that our $T$ is a `bound-covering mapping', see \cite{es}).
Since $C[0,1]$ is universal in the class of all separable Banach spaces, by \cite[Proposition 1.5]{Pelczynski-Bessaga},
the Banach space $\ell_{1}$ (isometrically) embeds into $C[0,1]$. Hence $\ell_1$  embeds into $C_k(X)$, which contradicts (3).
\smallskip

(2)$\Longrightarrow$(3)  Let (2) hold. Assume that $C_{k}(X)$ contains an isomorphic copy of $\ell_{1}$. By Lemma \ref{ku} there exists
a compact (scattered) subset $K\subset X$ such that $\ell_1$ is isomorphic to a subspace of the Banach space $C(K)$. But this contradicts Theorem \ref{1}.
Hence $C_k(X)$ does not contain any isomorphic  copy of $\ell_1$.

\smallskip
(4)$\Longrightarrow$(3) It is true since the dual of $\ell_1$ is isometric to $\ell_\infty$, which is non-separable.

\smallskip

(2)$\Longrightarrow$(4) It follows from Theorem~\ref{1} and Lemma~\ref{ku}.
\end{proof}
\begin{remark}
Theorem~\ref{MAIN} admits several natural refinements and consequences.

\smallskip

\noindent
(a) The assertion~(1) of Theorem~\ref{MAIN} remains valid in the following stronger form:

\medskip

\centerline{\sl For every Asplund Banach space $E$, the product $C_k(X)\times E$ is Asplund.}

\medskip

Indeed, the proof of the implication $(2)\Rightarrow(1)$ carries over almost verbatim. One only has to replace $C_k(X)$ by $C_k(X)\times E$. In particular, the mapping
\[
(f,g)\in C_k(X)\times E \longmapsto (f|_L,g)\in C(L)\times E
\]
retains the crucial bounded-covering property enjoyed by the operator $T$. Since $C(L)\times E$ is a product of Asplund Banach spaces, it is Asplund by \cite{np}. An alternative verification may be found in \cite{es}.

\smallskip

\noindent
(b) Another consequence concerns products of function spaces. If every compact subset of both spaces $X$ and $Y$ is scattered, then the product
\[
C_k(X)\times C_k(Y)
\]
is Asplund. Indeed, it is well known that
\[
C_k(X)\times C_k(Y)\cong C_k(X\oplus Y),
\]
where $X\oplus Y$ denotes the topological sum of $X$ and $Y$. Since every compact subset of $X\oplus Y$ is scattered whenever the same holds in $X$ and $Y$, Theorem~\ref{MAIN} applies directly.
\end{remark}

The preceding observations naturally lead to the following question, which remains open.

\begin{problem}
Assume that $C_k(X)$ is Asplund. Must the product $C_k(X)\times E$ be Asplund for every Asplund locally convex space $E$?
\end{problem}

A closer inspection of the proof of the implication $(2)\Rightarrow(1)$ in Theorem~\ref{MAIN} yields the following result, tightly related to Theorem~\ref{Phelps}; see also \cite[Theorem~3.2 and Example~3.6]{es}.

\begin{theorem}\label{first}
If $X$ is a Tychonoff space and $\vv$ is a continuous convex function on a nonempty  convex open  $D\subset C_k(X)$,
the (possibly empty) set of points of Fr\'echet differentiability of $\vv$ is a $G_{\delta}$-set.
\end{theorem}

\section{$\Delta_1$-spaces $X$ and the Asplund property for spaces  $C_k(X)$}
In this section we briefly outline the proof of Theorem~\ref{mainIII}, originally established in a slightly stronger form in \cite[Theorem~3.2]{KKuL}. The argument shows the significance of the class of $\Delta_1$-spaces introduced and systematically investigated in \cite{KKuL}. These spaces provide a remarkable link between topological covering properties and differentiability theory. In particular, every compact subset of a $\Delta_1$-space is scattered, and therefore Theorem~\ref{MAIN} immediately implies that the corresponding function space $C_k(X)$ is an Asplund locally convex space.

Beyond its role in the proof of Theorem~\ref{mainIII}, the $\Delta_1$-property yields further insight into the structure of spaces of continuous functions. In particular, it leads to an alternative topological characterization of compact spaces $X$ for which the Banach space $C(X)$ is Asplund, thus offering  a new perspective on a classical theorem.

To place these results in context, we recall several compactness-type properties that play an important role in what follows.

\begin{itemize}
\item A space $X$ is called {\it $\omega$-bounded} if the closure of every countable subset of $X$ is compact.

\item A space $X$ is called {\it countably compact} if every countably infinite subset of $X$ has an accumulation point.

\item A space $X$ is called {\it pseudocompact} if every continuous real-valued function on $X$ is bounded.
\end{itemize}

These notions form a natural hierarchy:
\[
\omega\text{-bounded}
\ \Longrightarrow\
\text{countably compact}
\ \Longrightarrow\
\text{pseudocompact}.
\]

Numerous examples separating these classes, as well as related compactness-like properties, may be found in \cite{KKuL}. We only mention two particularly instructive examples.

The ordinal interval $[0,\omega_1)$ endowed with the order topology is a classical example of a locally compact, noncompact, scattered space that is nevertheless $\omega$-bounded. On the other hand, the Tychonoff plank
\[
[0,\omega_1]\times[0,\omega]\setminus\{(\omega_1,\omega)\}
\]
is a scattered pseudocompact space which fails to be $\omega$-bounded. These examples illustrate the substantial gap between the compactness-type properties introduced above.

\begin{theorem} \cite{KKuL} \label{mainIII}
For a Tychonoff $\omega$-bounded space $X$ the following assertions are equivalent:
\begin{enumerate}
\item $X$ is scattered.
\item Every compact subset of $X$ is scattered.
\item $X$ is a $\Delta_1$-space.
\item $C_k(X)$ is an Asplund space.
\end{enumerate}
\end{theorem}
Prior to presenting the proof of Theorem~\ref{mainIII}, we introduce a number of concepts and preparatory results that play a central role in the argument.
\begin{definition} \cite{KKuL1}, \cite{KL1}
A topological  space $X$ is called a $\Delta$-space  ($\Delta_1$-space) if for every decreasing sequence $(D_n)_{n\in\omega}$
of (countable, respectively) subsets of $X$ with $\bigcap_{n\in\omega}D_n=\emptyset$, there is a decreasing sequence of open sets $(V_n)_{n\in\omega}$ such that $D_n \subset V_n$
for each $n\in\omega$,
and again with $\bigcap_{n\in\omega}V_n=\emptyset$.
\end{definition}
The notion of a $\Delta$-set of the real line was introduced by M. Reed and E. van Douwen, see \cite{Reed}. Since its formulation, the corresponding $\Delta$- and $\Delta_1$-properties have proved to be surprisingly fruitful in several areas of general topology. In particular, they have found important applications in the theory of function spaces $C_p(X)$  endowed with the topology of pointwise convergence, leading to a deeper understanding of the structure of spaces of the form $C_p(X)$; see \cite{KKuL1, KL1}.

A further reason for the interest in the $\Delta_1$-property is that it naturally extends one of the classical notions of descriptive set-theoretic topology, namely that of a $\lambda$-set. Recall that a subset $X\subset\mathbb{R}$ is called a $\lambda$-set if every countable subset of $X$ is a $G_\delta$-subset of $X$. More generally, a topological space $X$ is called a $\lambda$-space if every countable subset of $X$ is a $G_\delta$-set. The study of $\lambda$-sets goes back to a celebrated result of K. Kuratowski from 1933, who proved within ZFC the existence of uncountable $\lambda$-sets; see \cite[Chapter~3, \S40.III]{k}.

The class of $\Delta_1$-spaces is strictly broader than the class of $\lambda$-spaces, since every $\lambda$-space possesses the $\Delta_1$-property. Thus, $\Delta_1$-spaces establish a natural framework in which many classical phenomena related to $\lambda$-sets can be studied in a more general setting. At the same time, the two notions coincide within the realm of metrizable spaces, i.e., a metrizable space has the $\Delta_1$-property if and only if it is a $\lambda$-space, by \cite[Theorem~2.19]{KKuL1}.

First we present a proof of the equivalence (1) and (2) in Theorem \ref{mainIII}.
\begin{lemma}\label{ten}
Let $X$ be a Tychonoff $\omega$-bounded space. Then every compact subset of $X$ is scattered
if and only if  $X$ is scattered.
\end{lemma}
\begin{proof}
Assume that every compact subset of $X$ is scattered and suppose, contrary to the claim, that $X$ is not scattered. Since $X$ is pseudocompact, \cite[Proposition~5.5]{Leiderman-Tka} yields a closed subset $F\subset X$ and a continuous surjection $h:F\to[0,1]$.

Let $Q\subset[0,1]$ be countable and dense, and choose a countable set $C\subset F$ with $h(C)=Q$. Since $F$ is $\omega$-bounded, the closure $A:=\overline{C}$ is compact. Furthermore,
\[
h(A)\supset\overline{h(C)}=\overline{Q}=[0,1],
\]
so $h(A)=[0,1]$.

The compact space $A$ is scattered by assumption, yet it can be mapped continuously onto $[0,1]$, contradicting Theorem~\ref{sca}. Hence $X$ is scattered. The converse implication is obvious.
\end{proof}

For the reader's convenience, below we collect several most important results on $\Delta_1$-spaces.
Despite the close relationship between a $\Delta_1$-space $X$ and the question of whether $C_k(X)$ is an Asplund space,
 we omit the proofs of the statements presented here. These results are of a purely topological nature and lie somewhat outside the main scope of this work.
\begin{theorem}\label{del-1}\hfill
\begin{itemize}
\item [(a)] \cite[Theorem 2.15]{KKuL1}
Every regular scattered space is a $\Delta_1$-space.
\item [(b)]  \cite[Corollary 2.16]{KKuL1}
A compact space $X$ is a $\Delta_1$-space if and only if $X$ is scattered.
\item [(c)]  \cite[Theorem 2.9]{KKuL1}
A pseudocompact space $X$ is a $\Delta_1$-space if and only if every countable subset of $X$ is scattered.
\item [(d)]  \cite[Theorem 3.9]{KKuL1}
Let $X$ be a countable union of closed subsets $X_n$. If each $X_n$ is a $\Delta_1$-space
then $X$ also is a $\Delta_1$-space.
\end{itemize}
\end{theorem}

We are ready to prove our main Theorem  \ref{mainIII} of this section.

\begin{proof}[Proof of Theorem \ref{mainIII}]

(1)$\Longleftrightarrow$(2) This is Lemma \ref{ten}.

(1)$\Longrightarrow$(3)  Trivially follows from Theorem \ref{del-1}(a).

(3)$\Longrightarrow$(2) Assume  that $X$ is a $\Delta_1$-space. Then every compact subset of $X$ is a $\Delta_1$-space.
We apply Theorem \ref{del-1}(b).
\end{proof}
\begin{remark}
The pseudocompactness assumption in Theorem~\ref{mainIII} cannot be weakened to countable compactness. Indeed, there exists a countably compact, non-scattered dense subspace
\[
X\subset \beta\omega\setminus\omega
\]
such that every countable subset of $X$ is scattered. Consequently, $X$ is a $\Delta_1$-space. Moreover, every compact subset of $X$ is finite; see \cite[Example~2.18]{KKuL1}. Thus, even very strong scatteredness properties of compact subsets do not suffice to guarantee the conclusion of Theorem~\ref{mainIII} under the weaker assumption of countable compactness.
\end{remark}

\begin{remark}
The condition that every compact subset of $X$ be scattered is, by itself, substantially weaker than the $\Delta_1$-property.

A classical illustration is given by a Bernstein set $\B\subset\mathbb R$, that is, a set for which both $\B$ and $\mathbb R\setminus\B$ meet every uncountable closed subset of $\mathbb R$. It follows that every compact subset of $\B$ is countable and hence scattered. Nevertheless, $\B$ fails to be a $\Delta_1$-space; see \cite[Example~2.20]{KKuL1}.

The situation remains equally complex within the class of pseudocompact spaces. In fact, there exists a pseudocompact space $X$ such that every compact subset of $X$ is scattered, whereas $X$ is not a $\Delta_1$-space; see \cite[Example~3.2]{KKuL1}.

These examples demonstrate that the $\Delta_1$-property imposes significantly stronger restrictions on the topology of a space than the mere scatteredness of its compact subspaces.
\end{remark}
We conclude this section with a strengthened form of the Namioka--Phelps theorem, where condition~(2) is shown to be equivalent to the classical formulation.
\begin{theorem}
For a compact space $X$ the following assertions are equivalent:
\begin{enumerate}
\item $X$ is scattered.
\item $X$ is a $\Delta_1$-space.
\item The Banach space $C(X)$ is Asplund.
\end{enumerate}
\end{theorem}

\section{Illustrative examples and the role of the strongest locally convex topology}
It is a classical and highly nontrivial fact that every closed linear subspace of an Asplund Banach space is again Asplund; see, for instance, \cite{p} and \cite[Proposition~2.33]{np}. Surprisingly, this permanence property breaks down in the broader setting of locally convex spaces. Counterexamples were obtained in \cite{Kakol-Leiderman}, and the construction relies on the remarkable space~$\varphi$.

Recall that $\varphi$ denotes a countably dimensional vector space endowed with the finest locally convex topology $\xi_\infty$, i.e., the topology whose neighbourhood base at the origin consists of all absolutely convex absorbing subsets. Equivalently, $\xi_\infty$ is the strongest locally convex topology on the underlying vector space, and therefore every linear functional on $\varphi$ is automatically continuous.

A distinctive feature of $\varphi$ is that every bounded subset is finite-dimensional. Indeed, if a bounded set were not contained in a finite-dimensional subspace, the Hahn--Banach theorem would yield a discontinuous linear functional, contradicting the definition of $\xi_\infty$. On the other hand, $\varphi$ admits a natural representation as the strict countable inductive limit of finite-dimensional Banach spaces $\mathbb{R}^n$. More precisely, if $\xi$ denotes the locally convex inductive limit topology generated by the increasing sequence of finite-dimensional subspaces, then $\xi=\xi_\infty$; see \cite[Chapter~2.4]{KKPS} and \cite[0.3.1]{Bonet}. Moreover, this topology is complete; see \cite[Proposition~8.4.16(iv)]{Bonet} or \cite{Zelazko}. The latter fact will be  used in the final part of the proof of Proposition \ref{UU}.

An important observation due to B. Sharp \cite{s} is that $\varphi$ fails to be a GDS space. To see this, consider the $\ell_1$-norm $\|\cdot\|_1$, which is continuous with respect to $\xi_\infty$ and therefore defines a continuous convex function on $\varphi$. Nevertheless, $\|\cdot\|_1$ is nowhere G\^ateaux differentiable; see \cite[Example~1.4(b)]{p}. Indeed, for every $x\in\varphi$ there exists $n\in\mathbb N$ such that the $n$-th coordinate of $x$ vanishes. Let
\[
e_n=(0,\ldots,0,1,0,\ldots)
\]
be the $n$-th canonical unit vector. Then
\[
\lim_{t\downarrow0}
\frac{\|x+te_n\|_1-\|x\|_1}{t}=1,
\qquad
\lim_{t\uparrow0}
\frac{\|x+te_n\|_1-\|x\|_1}{t}=-1,
\]
showing that the directional derivative does not exist at $x$. Since $x$ was arbitrary, the norm $\|\cdot\|_1$ is nowhere G\^ateaux differentiable on $\varphi$.

\begin{remark}
The representation of $\varphi$ as the strict inductive limit of the spaces $\mathbb R^n$ implies that $\varphi$ may be identified with the free locally convex space $L(\mathbb N)$. More generally, it was proved in \cite{Kakol-Leiderman} that the free locally convex space $L(X)$ is not GDS whenever $X$ is any infinite Tychonoff space.
\end{remark}

To establish Proposition~\ref{UU}, we shall make use of two outstanding results due to S. Saxon \cite{Saxon}. Since some steps in the original arguments are only sketched, we include complete proofs for the reader's convenience. We begin by reproving the following technical lemma from \cite[Lemma]{Saxon}.

\begin{proposition}
Let $E$ be a locally convex space. Assume that  $(A_n)$ is an increasing sequence of closed absolutely convex subsets of $E$ covering $E$
and that there exists a sequence $(x_n)$ such that
\[
x_n\in A_{n+1}\setminus \operatorname{sp}(A_n)
\qquad (n\in\mathbb N).
\]

Then the dual space $E'$ contains functionals $f_n\in A_n^\circ$ $(n\in\mathbb N)$ for which
\[
q(x):=\sup_{n\in\mathbb N}|f_n(x)|
\]
satisfies
\[
\sum_{i=1}^{\infty}|a_i|
\le
q\!\left(\sum_{i=1}^{\infty}a_i x_i\right)
\qquad\text{for all }(a_i)\in c_{00}.
\]
\end{proposition}
\begin{proof}
The idea is to construct inductively a sequence of continuous functionals
\[
f_n\in A_n^\circ,\qquad n\in\mathbb N,
\]
whose associated seminorms dominate the $\ell_1$-norm of the coefficients of finite linear combinations of the vectors $x_n$.

For every $k\in\mathbb N$ define
\[
q_k(x):=\max\{|f_1(x)|,\ldots,|f_k(x)|\}.
\]
We shall show inductively that
\begin{equation}\label{eq:induction}
\frac{k+1}{k}\sum_{i=1}^{k}|a_i|
   \le
q_k\Big(\sum_{i=1}^{k}a_ix_i\Big)
\qquad
\forall (a_i)\in c_{00}.
\end{equation}
Since $(k+1)/k\to1$, letting $k\to\infty$ will yield
\[
\sum_{i=1}^{\infty}|a_i|
   \le
q\Big(\sum_{i=1}^{\infty}a_ix_i\Big),
\]
where
\[
q(x):=\sup_{n\in\mathbb N}|f_n(x)|.
\]

\medskip

We start with the construction of $f_1$. Since
$x_1\notin 2 A_1$
and $2 A_1$ is a closed convex set, the Hahn--Banach separation theorem yields a functional
$f_1\in E'$ such that
\[
f_1(x_1)=2
   >
\sup f_1\Big[2 A_1\Big].
\]
Consequently,
$f_1\in A_1^\circ$.
Moreover,
\[
q_1(a_1x_1)=|a_1|\,|f_1(x_1)|
=2 |a_1|,
\]
and therefore \eqref{eq:induction} holds for $k=1$.

\medskip

Assume now that for some $k\in\mathbb N$ we have already constructed
\[
f_1\in A_1^\circ,\ldots,f_k\in A_k^\circ
\]
and that \eqref{eq:induction} holds.

Choose $P_0>0$ so large that
\begin{equation}\label{eq:P0}
\frac{k+2}{k+1}
<
\frac{(1+1/k)P-q_k(x_{k+1})}{P+1}
\qquad
\forall P\ge P_0.
\end{equation}

Set
\begin{equation}\label{eq:Q}
Q:=
\frac{k+2}{k+1}(P_0+1)+P_0.
\end{equation}

Since
$x_{k+1}\notin {\rm sp}(A_k)$
and
${\rm sp}(A_k)\subset QA_{k+1}$,
the Hahn--Banach theorem provides
$f_{k+1}\in E'$
such that
\begin{equation}\label{eq:HB}
f_{k+1}(x_{k+1})=Q
>
\sup f_{k+1}[QA_{k+1}].
\end{equation}
Hence
\[
f_{k+1}\in A_{k+1}^\circ.
\]

Fix
\[
x:=\sum_{i=1}^{k}a_ix_i .
\]
Since
$x_1,\ldots,x_k\in A_{k+1}$,
we obtain
\begin{equation}\label{eq:estimate21}
|f_{k+1}(\pm x)|
\le
\sum_{i=1}^{k}|a_i|.
\end{equation}

We distinguish two cases.

\medskip

\noindent
{\bf Case 1.}
\[
\sum_{i=1}^{k}|a_i|\ge P_0.
\]

Using \eqref{eq:P0}, the inductive hypothesis
\eqref{eq:induction}, and the triangle inequality, we get
\[
\frac{k+2}{k+1}
<
\frac{(1+1/k)\sum_{i=1}^{k}|a_i|-q_k(x_{k+1})}
     {\sum_{i=1}^{k}|a_i|+1}
\]
\[
\le
\frac{
q_k\!\left(\sum_{i=1}^{k}a_ix_i\right)-q_k(x_{k+1})
}{
\sum_{i=1}^{k}|a_i|+1
}
\le
\frac{
q_k(\pm x+x_{k+1})
}{
\sum_{i=1}^{k}|a_i|+1
}
\]
\[
\le
\frac{
q_{k+1}(\pm x+x_{k+1})
}{
\sum_{i=1}^{k}|a_i|+1
}.
\]
Hence
\begin{equation}\label{eq:case1}
\frac{k+2}{k+1}
\Big(
\sum_{i=1}^{k}|a_i|+1
\Big)
<
q_{k+1}(\pm x+x_{k+1}).
\end{equation}

\medskip

\noindent
{\bf Case 2.}
\[
\sum_{i=1}^{k}|a_i|<P_0.
\]

Using \eqref{eq:Q}, \eqref{eq:HB}, and
\eqref{eq:estimate21}, we obtain
\[
\frac{k+2}{k+1}
=
\frac{Q-P_0}{P_0+1}
<
\frac{
Q-\sum_{i=1}^{k}|a_i|
}{
\sum_{i=1}^{k}|a_i|+1
}
\]
\[
=
\frac{
f_{k+1}(x_{k+1})
-\sum_{i=1}^{k}|a_i|
}{
\sum_{i=1}^{k}|a_i|+1
}
\le
\frac{
f_{k+1}(x_{k+1})
-|f_{k+1}(\pm x)|
}{
\sum_{i=1}^{k}|a_i|+1
}
\]
\[
\le
\frac{
|f_{k+1}(\pm x+x_{k+1})|
}{
\sum_{i=1}^{k}|a_i|+1
}
\le
\frac{
q_{k+1}(\pm x+x_{k+1})
}{
\sum_{i=1}^{k}|a_i|+1
}.
\]
Therefore,
\begin{equation}\label{eq:case2}
\frac{k+2}{k+1}
\Big(
\sum_{i=1}^{k}|a_i|+1
\Big)
\le
q_{k+1}(\pm x+x_{k+1}).
\end{equation}

\medskip

In both cases we arrive at
\[
\frac{k+2}{k+1}
\Big(
\sum_{i=1}^{k}|a_i|+1
\Big)
\le
q_{k+1}(\pm x+x_{k+1}).
\]
Choosing the sign according to the sign of $a_{k+1}$ and replacing
$x$ by $\sum_{i=1}^{k}a_ix_i$, we obtain
\[
\frac{k+2}{k+1}
\sum_{i=1}^{k+1}|a_i|
\le
q_{k+1}
\Big(
\sum_{i=1}^{k+1}a_ix_i
\Big).
\]
Thus the inductive step is complete.

Hence \eqref{eq:induction} holds for every $k\in\mathbb N$. Passing to the limit as $k\to\infty$, we conclude that
\[
\sum_{i=1}^{\infty}|a_i|
\le
q\Big(\sum_{i=1}^{\infty}a_ix_i\Big)
\qquad
\forall (a_i)\in c_{00},
\]
which completes the proof.
\end{proof}
Having established the necessary preliminaries, we now turn to the following remarkable theorem due to S. Saxon \cite[Theorem~2.1]{Saxon}, which will serve as a key tool in what follows.

Let us first recall two important notions that we will use: An lcs $E$ is \emph{ barrelled } if every closed absolutely convex absorbing subset of $E$ is a neighbourhood of zero.  An lcs $E$ is called \emph{Baire-like} if for every increasing sequence $(A_n)_{n}$ of absolutely  convex  closed subsets of $E$ covering $E$ there exists $m\in\mathbb{N}$ such that $A_m$ is a neighbourhood of zero. Clearly
\begin{center}
 Baire $\Rightarrow$ Baire-like  $\Rightarrow$ barrelled, see \cite{Saxon} and \cite{Na-Saxon}.

\end{center}
\begin{theorem}\label{ZZ}
Let $E$ be a barrelled locally convex space which is not Baire-like. Then the space $\varphi$ embeds into $E$.
\end{theorem}

\begin{proof}
Since $E$ is not Baire-like, there exists an increasing sequence of closed, absolutely convex  sets
\[
0\in A_1\subsetneq A_2\subsetneq \cdots \subset E
\]
such that
\[
E=\bigcup_{n=1}^{\infty}A_n
\]
and
\[
{\rm span}(A_1)\subsetneq {\rm span}(A_2)\subsetneq\cdots.
\]
Consequently, for every $n\in\mathbb N$ we may choose
\[
z_n\in A_{n+1}\setminus {\rm span}(A_n).
\]
The vectors $z_1,z_2,\ldots$ are necessarily linearly independent.

Let
\[
S:={\rm span}\{z_1,z_2,\ldots\}\subset E.
\]
We equip $S$ with the finest locally convex topology and identify the resulting space with $\varphi$. Thus
$\{z_n:n\in\mathbb N\}$ becomes a Hamel basis of $\varphi$, and every
$z\in\varphi$
admits a unique representation
$z=\sum_{i=1}^{\infty}a_i z_i,
\qquad (a_i)\in c_{00}.$

To prove that $\varphi$ embeds into $E$, it suffices to show that for every continuous seminorm
\[
p:\varphi\to [0,\infty)
\]
there exists a continuous seminorm
$\rho:E\to[0,\infty)$
such that
\[
p(z)\le \rho(z)
\qquad (z\in\varphi).
\]
Indeed, the canonical injection
\[
j:\varphi\to E,\qquad j(z)=z,
\]
is automatically continuous, and the above domination property implies the continuity of the inverse map
$j^{-1}:S\to\varphi$.

Fix a continuous seminorm $p$ on $\varphi$ and set
\[
I:=\{i\in\mathbb N:\ p(z_i)\le 1\}.
\]
For every $i\in\mathbb N$ define
\[
x_i=
\begin{cases}
z_i, & i\in I,\\[2mm]
\dfrac{z_i}{p(z_i)}, & i\notin I.
\end{cases}
\]
Then
\[
p(x_i)\le 1
\qquad\text{for all }i\in\mathbb N,
\]
and
\[
x_n\in A_{n+1}\setminus {\rm span}(A_n),
\qquad n\in\mathbb N.
\]
Moreover,
\[
{\rm span}\{x_1,x_2,\ldots\}=S.
\]

Applying the preceding proposition to the sequences $(A_n)$ and $(x_n)$, we obtain functionals
\[
f_n\in A_n^\circ,\qquad n\in\mathbb N,
\]
and a seminorm
\[
q(x):=\sup_{n\in\mathbb N}|f_n(x)|
\]
such that
\begin{equation}\label{eq:l1estimate}
\sum_{i=1}^{\infty}|a_i|
\le
q\Big(\sum_{i=1}^{\infty}a_i x_i\Big)
\qquad
\forall (a_i)\in c_{00}.
\end{equation}

Let
\[
z=\sum_{i=1}^{\infty}b_i z_i\in \varphi .
\]
Using the definition of the vectors $x_i$, the subadditivity of $p$, and \eqref{eq:l1estimate}, we obtain
\[
\begin{aligned}
p(z)
&=
p\Big(
\sum_{i\in I} b_i x_i
+
\sum_{i\notin I} p(z_i)b_i x_i
\Big) \\[1mm]
&\le
\sum_{i\in I}|b_i|
+
\sum_{i\notin I}|b_i|p(z_i)\\[1mm]
&\le
q\Big(
\sum_{i\in I} b_i x_i
+
\sum_{i\notin I} p(z_i)b_i x_i
\Big)\\[1mm]
&=
q\Big(
\sum_{i=1}^{\infty}b_i z_i
\Big)
=q(z).
\end{aligned}
\]
Hence
\[
p(z)\le q(z)
\qquad (z\in\varphi).
\]

It remains to verify that $q$ is continuous on $E$. Fix $x\in E$. Since
$x\in A_m$ for some $m$, we have
\[
|f_n(x)|\le 1
\qquad \text{for all }n\ge m,
\]
because $f_n\in A_n^\circ$ and $A_m\subset A_n$. Thus the sequence
$(f_n)$ is pointwise bounded on $E$. Since $E$ is barrelled, the Banach--Steinhaus theorem implies that
$\{f_n:n\in\mathbb N\}$ is equicontinuous. Therefore
\[
q(x)=\sup_{n\in\mathbb N}|f_n(x)|
\]
is a continuous seminorm on $E$.

We have shown that every continuous seminorm on $\varphi$ is dominated by the restriction of a continuous seminorm on $E$. Consequently, the inverse mapping
$j^{-1}:S\to\varphi$
is continuous, and hence the canonical injection $\varphi\hookrightarrow E$
is a topological embedding.
\end{proof}

We now combine the previous results to derive the following striking example from \cite{Kakol-Leiderman}. It demonstrates that several permanence properties familiar from the theory of Asplund Banach spaces may fail dramatically in the broader setting of locally convex spaces.

In what follows, $Q$ stands for the space of rational numbers with the topology induced from $\mathbb{R}$. Consequently, $Q$ is countable and metrizable.
\begin{proposition} \cite{Kakol-Leiderman} \label{UU}
The space $C_k(Q)$ over the rationals $Q$ is an Asplund space but isomorphically contains a closed copy of the non-Asplund space $\varphi$.
\end{proposition}
\begin{proof}
Every compact subset of $Q$ is countable, hence it is scattered. By Theorem \ref{MAIN} we deduce that $C_k(Q)$ is an Asplund space. Since $Q$ is a $\mu$-space (i.e., every functionally bounded subset of $Q$ is relatively compact), the space $C_k(Q)$ is barrelled by the Nachbin--Shirota theorem, see \cite[Propostion 2.4.16]{KKPS}, meaning that every closed absolutely convex absorbing subset of $C_k(Q)$ is a neighbourhood of zero.

On the other hand, the space $C_k(Q)$ is not Baire-like, i.e., there exists an increasing sequence of closed convex absorbing subsets $A_1, A_2,\ldots$ covering the space $C_k(Q)$ such that no $A_n$ is a neighbourhood of the origin in $C_k(Q)$.

Indeed, assume on the contrary that $C_k(Q)$ is a Baire-like space. Then we follow the argument contained in the proof of \cite[Proposition 2.4.20]{KKPS},
or \cite[Corollary 5.3.4]{mccoy}: For $n\in\N$ we put
$$
U_n:= (-1/n,1/n)\cap Q\quad {\rm and}\quad H_n:= \big\{f\in C(Q):\  f[U_n]\subset [-n,n]\big\}.
$$
Clearly, the $U_n$'s form a basis of neighbourhoods of $0$ in $Q$ and the (closed convex and absorbing) sets $H_n$'s cover all of $C_k(Q)$.
Also, clearly, $$U_1 \supset U_2\supset \cdots, \,\,\,H_1\subset H_2\subset \cdots.$$

Since we assumed that $C_{k}(Q)$ is Baire-like, there exist a compact
set $K\subset Q$, $\varepsilon>0$, and $n\in\N$ such that
$$
\{f\in C(X):\ f[K] \subset (-\varepsilon,\varepsilon)\}\subset H_{n}.
$$
Note that $U_{n}\subset K$. Indeed, if there is a $z\in U_{n}\setminus K$, then putting
$$
f(x):= 2n\frac{{\rm dist}(x,K)}{{\rm dist}(z,K)},\quad x\in Q,
$$
we have that $f\in C_k(Q)$ and that $f(z)=2n$. Also $f_{|K}\equiv 0$, and thus $f\in U_n$, which is impossible.
Therefore, we have that $U_{n}\subset K$. This leads to a contradiction because $K$ is a countable compact subset of $\R$ and the closure of $U_{n}$ in $\R$ is an uncountable closed interval.

We have proved that the space $C_k(Q)$ is not Baire-like.
Finally, since $C_k(Q)$ is barrelled but not Baire-like, we apply Saxon's Theorem \ref{ZZ}  and conclude that the locally convex space  $C_k(Q)$ contains isomorphically a (closed) copy of the space $\varphi$.
\end{proof}
\begin{remark}
The locally convex space $\varphi$ is nonmetrizable and complete, since it can be represented as the strict countable inductive limit of finite-dimensional spaces. On the other hand, its strong dual $\varphi^{*}$, endowed with the topology of uniform convergence on bounded subsets of $\varphi$, is isomorphic to the countable product $\mathbb{R}^{\omega}$ and is therefore separable. Thus, despite possessing a separable strong dual, the space $\varphi$ fails to be Asplund.
\end{remark}
\section{The space $C_p(X)$  is always an Asplund space}
In this last section  we  give a short proof of Sharp's  Theorem  \cite[Theorem 5.5]{s} stating that a locally convex space endowed with its weak topology is an  Asplund space.  The whole argument is the finite-dimensional factorisation of weakly continuous convex functions through a finite face; convexity then forces the function to be constant on the common kernel of the factorization mapping throughout the whole convex domain.
Further, the proof of Theorem \ref{sharp} below is very similar to Sharp's theorem on weak topologies.

Since the space $C_p(X)$ of real-valued continuous functions on a Tychonoff space $X$  endowed with the pointwise topology carries its weak topology, see \cite[Section I.2]{Schmets}, the above result shows  that every space $C_p(X)$, endowed with the topology of pointwise convergence, possesses the Asplund property for any Tychonoff space $X$.  Moreover, when compared with the corresponding theorem for spaces $C_k(X)$, this result reveals a significant difference in the behavior of the Asplund property within the two classes of function spaces $C_k(X)$ and $C_p(X)$. While the Asplund property for $C_k(X)$ generally requires additional assumptions on the underlying Tychonoff space $X$, in the case of $C_p(X)$ it holds for any  Tychonoff  space $X$.

We prove the following main theorem due to B. Sharp \cite[Theorem 5.5]{s}.
\begin{theorem}\label{sharp}
Let $E$ be a locally convex space and let $E_{w}:=(E,\sigma(E,E'))$.  Then $E_{w}$ is an Asplund space.
\end{theorem}
The proof proceeds through the following stronger structural statement.
\begin{theorem}\label{sharp2}
Let $D\subset E_{w}$ be nonempty, open and convex, and let $f\colon D\to\R$ be continuous and convex.  Then there are a finite-dimensional locally convex space $Y$, a continuous open linear map
$T\colon E_{w}\longrightarrow Y,$ an open convex set $\Omega\subset Y$, and a continuous convex function $g\colon\Omega\to\R$ such that
$D=T^{-1}(\Omega),\,\,f=g\circ T\quad\text{on }D.$
\end{theorem}
Theorem \ref{sharp} is therefore an extension of the classical  fact that continuous convex functions on open convex subsets of finite-dimensional spaces are Fr\'echet differentiable on dense $G_\delta$ sets.

We will need two additional lemmas.  The point of the weak topology is that every basic neighbourhood depends only on finitely many linear coordinates.  If a convex function is bounded above on a translate of the common kernel of those coordinates, then it must be constant there.  Convexity and openness of the domain imply this property from one translate to all translates that meet the domain.
\begin{lemma}\label{le1}
Let $C$ be a convex subset of a vector space $X$, let $W$ be a linear subspace of $X$, let $x+W\subset C$ for some $x\in X$,  and let $f\colon C\to\R$ be convex.  If $f$ is bounded above on $x+W$, then $f$ is constant on $x+W$.
\end{lemma}
\begin{proof}
Let $m$ be an upper bound for $f$ on $x+W$.  Take arbitrary $y,z\in x+W$ and $r>1$.  Since $y+r(z-y)\in x+W$ and
$$z=\frac1r\bigl(y+r(z-y)\bigr)+\left(1-\frac1r\right)y,$$
convexity gives
\[
        f(z)\leq \frac{m}{r}+\left(1-\frac1r\right)f(y).
\]
Letting $r\to\infty$ yields $f(z)\leq f(y)$.  By symmetry $f(y)\leq f(z)$, so $f(y)=f(z)$.
\end{proof}
\begin{lemma}\label{le2}
Let $Y$ be finite-dimensional, let $\Omega\subset Y$ be open and convex, and let $g\colon\Omega\to\R$ be  convex.  Then the set of points at which $g$ is Fr\'echet differentiable contains a dense $G_\delta$ subset of $\Omega$.
\end{lemma}
\begin{proof}
First note that $g$ is continuous, see \cite[page 11]{p}. After choosing a linear isomorphism $Y\cong\R^n$, this is the standard finite-dimensional theorem for convex functions.  Indeed, continuous convex functions on open convex subsets of $\R^n$ are even locally Lipschitz \cite[pp. 4,5]{p}; Mazur's theorem \cite[Theorem 1.20]{p}  gives differentiability almost everywhere, hence on a dense set, and the differentiability set of a finite-valued convex function is a $G_\delta$ set.  In finite-dimensional spaces the usual derivative and the locally convex Fr\'echet derivative above coincide.
\end{proof}
\begin{proof}[Proof of Theorem \ref{sharp2}]
Choose $x_0\in D$.  Since $D$ is weakly open and $f$ is weakly continuous at $x_0$, there are $\varphi_1,\ldots,\varphi_n\in E'$, $\epsilon>0$, and
\[
        V=\{u\in E: |\varphi_i(u)|<\epsilon\ (i=1,\ldots,n)\}
\]
such that
\[
        x_0+V\subset D
        \quad\text{and}\quad
        f(x_0+u)< f(x_0)+1\qquad (u\in V).
\]
Let
\[
        T\colon E_{w}\longrightarrow Y:=T(E)\subset\R^n,
        \qquad
        T(x)=(\varphi_1(x),\ldots,\varphi_n(x)),
\]
and put $W=\ker T$.  The space $Y$ is finite-dimensional  and $T$ is a continuous open surjection onto $Y$.  To see the openness, factor $T$ through the quotient $E_{w}/W$; this quotient is a finite-dimensional Hausdorff locally convex space, hence carries the unique (finite-dimensional) topology.  Since $W\subset V$, the function $f$ is bounded above on $x_0+W$; by  Lemma \ref{le1} the function  $f$ is constant on $x_0+W$.

We now show that the same is true on every translate meeting $D$.  Fix $x\in D$.  Since $D$ is open, there is $t>0$ such that
\[
        z:=(1+t)x-tx_0\in D.
\]
Let $w\in W$ and set $u=(1+t)w/t\in W$.  Then
\[
        x+w=\frac{1}{1+t}z+\frac{t}{1+t}(x_0+u).
\]
Both $z$ and $x_0+u$ belong to $D$, so $x+w\in D$ by convexity of $D$.  Moreover, with $M=f(x_0)+1$, convexity gives the uniform estimate
\[
        f(x+w)
        \leq \frac{1}{1+t}f(z)+\frac{t}{1+t}M
        \qquad (w\in W).
\]
Thus $f$ is bounded above on $x+W$.  Lemma \ref{le1} implies that $f$ is constant on $x+W$.

Consequently, $D$ is saturated with respect to $W$: if $x\in D$ and $x-y\in W$, then $y\in D$; and $f$ is constant on each such fiber.  Put
\[
        \Omega:=T(D)\subset Y.
\]
Since $T$ is open, $\Omega$ is open; it is convex because $D$ is convex.  Define
\[
        g\colon\Omega\to\R,
        \qquad
        g(Tx)=f(x)\quad (x\in D).
\]
This mapping $g$ is well-defined because the constancy on fibers is just proved.  It is easy to see that the mapping $g$ is convex.  It is also continuous: if $O\subset\R$ is open, then
\[
        g^{-1}(O)=T\bigl(f^{-1}(O)\bigr),
\]
and $f^{-1}(O)$ is open in $E_{w}$ because $D$ itself is open; applying the openness of $T$ shows that $g^{-1}(O)$ is open in $Y$.  Finally, saturation gives $D=T^{-1}(\Omega)$, and by construction $f=g\circ T$.
\end{proof}
\begin{proof}[Proof of Theorem \ref{sharp}]
Let $f\colon D\to\R$ be continuous and convex, with $D\subset E_{w}$ nonempty, open and convex.  Apply Theorem \ref{sharp2} and write
\[
        f=g\circ T,
        \qquad
        D=T^{-1}(\Omega),
\]
where $Y$ is finite-dimensional, $\Omega\subset Y$ is open and convex, and $g\colon\Omega\to\R$ is continuous and convex.

By  Lemma \ref{le2}, let $G\subset\Omega$ be a dense $G_\delta$ set consisting of Fr\'echet differentiability points of $g$.  Then
\[
        A:=T^{-1}(G)\cap D=T^{-1}(G)
\]
is a $G_\delta$ subset of $D$, because $T$ is continuous and $D=T^{-1}(\Omega)$.

The set $A$ is dense in $D$.  Indeed, if $O$ is a nonempty relatively open subset of $D$, then $O$ is open in $E_{w}$, and hence $T(O)$ is a nonempty open subset of $\Omega$; since $G$ is dense in $\Omega$, we may choose $y\in T(O)\cap G$.  Pick $x\in O$ with $T x=y$.  Then $x\in O\cap A$.

It remains only to verify differentiability.  Let $x\in A$, and put $y=T x\in G$.  If $a\in Y'$ is the Fr\'echet derivative of $g$ at $y$, define
\[
        \ell:=a\circ T\in (E_{w})'.
\]
For every bounded set $B\subset E_{w}$, the set $T(B)$ is bounded in $Y$.  Hence the Fr\'echet differentiability of $g$ at $y$ gives
\[
\begin{aligned}
&\sup_{h\in B}
\left|\frac{f(x+t h)-f(x)}{t}-\ell(h)\right| \\
&\quad =
\sup_{h\in B}
\left|\frac{g(y+tT h)-g(y)}{t}-a(T h)\right|
\longrightarrow 0
\qquad (t\to0).
\end{aligned}
\]
Thus $f$ is Fr\'echet differentiable at every point of $A$.  Since $A$ is dense and $G_\delta$ in $D$, the space $E_{w}$ is Asplund.
\end{proof}

\end{document}